\documentclass[10pt]{amsart}

\usepackage{amsmath, amscd, amssymb, euler}
\usepackage[frame,cmtip,arrow,matrix,line,graph,curve]{xy}
\usepackage{graphpap, color,bbm, verbatim}
\usepackage[mathscr]{eucal}
\usepackage[normalem]{ulem}
\usepackage{pgf,tikz}
\usepackage{multirow,array}
\usetikzlibrary{arrows}

\numberwithin{equation}{section}

\newcommand{\CC}{\mathbb{C}}

\newcommand{\PP}{\mathbb{P}}

\newcommand{\ZZ}{\mathbb{Z}}

\newcommand{\cal}{\mathcal}

\def\cK{{\cal K}}

\def\cM{{\cal M}}

\def\cO{{\cal O}}

\def\cQ{{\cal Q}}

\def\cS{{\cal S}}

\def\cU{{\cal U}}

\newcommand{\lr}{\rightarrow}

\def\PP{\mathbb{P}}
\def\CC{\mathbb{C}}
\def\lra{\longrightarrow}

\newcommand{\ses}[3]{0\rightarrow{#1}\rightarrow{#2}\rightarrow{#3}\rightarrow0}

\newtheorem{prop}{Proposition}[section]
\newtheorem{theo}[prop]{Theorem}
\newtheorem{lemm}[prop]{Lemma}
\newtheorem{coro}[prop]{Corollary}
\newtheorem{rema}[prop]{Remark}
\newtheorem{defi}[prop]{Definition}

\def\beq{\begin{equation} }
\def\eeq{\end{equation} }
\def\zero{\mathrm{zero} }
\def\rank{\mathrm{rank} }

\title[Linear sections of Grassmannian]{Geometry of moduli spaces of rational curves in linear sections of Grassmannian $Gr(2,5)$}

\author{Kiryong Chung}
\address{Department of Mathematics Education, Kyungpook National University, 80 Daehakro, Bukgu, Daegu 41566, Korea}
\email{krchung@knu.ac.kr}
\author{Jaehyun Hong}
\address{School of Mathematics, Korea Institute for Advanced Study, 85 Hoegiro, Dongdaemun-gu, Seoul 02455, Republic of Korea}
\email{jhhong00@kias.re.kr}
\author{SangHyeon Lee}
\address{Department of Mathematical Sciences, Seoul National University, GwanAkRo 1, Seoul 08826, Korea}
\email{tlrehrdl@snu.ac.kr}

\keywords{Grassmannian, rational curves, Fano variety, birational geometry}
\subjclass[2010]{14E05, 14E08, 14M15.}

\begin{document}
\begin{abstract}
We prove that the moduli spaces of rational curves of degree at most $3$ in linear sections of the Grassmannian $Gr(2,5)$ are all rational varieties. We also study their compactifications and birational geometry.
\end{abstract}
\maketitle

\def\coindex{\mathrm{coindex} }
\def\index{\mathrm{index} }

\section{Introduction}
Rational curves in Fano varieties have played useful roles in algebraic geometry as in the works of Clemens-Griffiths \cite{CG}, Iskovski \cite{Isk}, Beauville-Donagai \cite{BD}, Lehn-Lehn-Sorger-van Straten \cite{LLSV}, Takkagi-Zucconi \cite{TZ11} and Iliev-Manivel \cite{iliev3}.
Understanding the birational geometry of the moduli spaces of rational curves in a Fano variety may lead us to interesting examples of new varieties or may reveal some internal structure of Fano varieties (\cite{CS09,PZ15}).
In \cite{K,CK,CHK}, the authors investigated the birational geometry of compactified moduli spaces of rational curves of degree $\le 3$ in projective spaces and homogeneous varieties.
The purpose of this paper is to investigate the geometry of moduli spaces of rational curves of degree $\le 3$ in linear sections of Grassmannian $Gr(2,5)$.

A smooth Fano varietiy refers to a smooth projective variety $Y$ whose canonical bundle $K_Y$ is antiample, i.e. the dual $K_Y^\vee$ of $K_Y$ is ample. The index of $X$ is defined to be the largest positive integer $r$ such that $K_Y^\vee$ is
the $r$-th power of a line bundle on $Y$. If $r=\index(Y)$, we denote the $r$-th root of $K_Y^\vee$ by $\cO_Y(1)$. It is well known \cite{IP} that the index of a smooth Fano variety is bounded from above by $\dim(Y)+1$. We define
$$\coindex(Y)=\dim Y+1-\index (Y).$$
If $\coindex(Y)=0$, $Y$ is a projective space. When $\coindex(Y)=1$, $Y$ is a quadric hypersurface, which is a homogeneous variety. Therefore the results of \cite{CK,CHK} account for the birational geometry of rational curves of degree $\le 3$ when $\coindex(Y)\le 1$.

Among Fano varieties of coindex 2, cubic hypersurfaces have attracted the most interest.
In \cite{CG}, Clemens-Griffiths utilized the moduli space of lines to prove the nonrationality of cubic 3-folds.
In \cite{BD}, Beauville-Donagi showed that the moduli space of lines in a cubic 4-fold is a hyperk\"ahler manifold of dimension 4.
In \cite{LLSV}, Lehn-Lehn-Sorger-van Straten investigated the birational geometry of moduli spaces of twisted cubic curves in a cubic 4-fold and constructed a hyperk\"ahler manifold of dimension 8.

Another interesting family of Fano varieties of coindex 2 is linear sections of Grassmannian $Gr(2,5)$.
Let $Gr(2,5)\subset \PP^9$ denote the Pl\"ucker embedding and let $Y^m_5$ denote the intersection of $Gr(2,5)$ with $6-m$ general hyperplanes in $\PP^9$ for $2\le m\le 6$. Then $Y^m_5$ is a smooth Fano $m$-fold of degree 5 and coindex 2.
These are of particular interest because they are solutions to Hirzebruch's problem: $Y^m_5$ is a smooth projective compactification of $\CC^m$ with $b_2(Y^m_5)=1$. The goal of this paper is to investigate the geometry of moduli spaces of rational curves in $Y^m_5$.

Rational curves in the Fano 3-fold $Y^3_5$ have been studied in a \cite{Fae, FuNa, Ili, San, CS16}.
Let $H_d(Y^3_5)$ denote the Hilbert scheme of closed subschemes of $Y^3_5$ with Hilbert polynomial $dt+1$. Then it was proved by Faenzi, Furushima-Nakayama, Iliev and Sanna that
\beq\label{sI1} H_1(Y^3_5)\cong \PP^2,\quad H_2(Y^3_5)\cong \PP^4,\quad \mathrm{and} \quad H_3(Y^3_5)\cong Gr(2,5).\eeq
In particular these moduli spaces are all rational and irreducible. In this paper, we prove
\begin{theo}\label{sI3} (Theorem \ref{H14})
The moduli spaces $R_d(Y^m_5)$ of smooth rational curves of degree $d$ on $Y^m_5$
are all rational for $d\le 3$ and $2\le m\le 6$.
\end{theo}
To prove this theorem, we first classify smooth rational curves on $Gr(2,5)$ of degree $\le 3$ (cf. \S\ref{sG}). We find that a line (resp. conic, resp. twisted cubic) in $Gr(2,5)$ has a \emph{vertex} (resp. \emph{envelope}, resp. \emph{axis}) which gives us a rational map
$\eta_1:R_1(Y^m_5)\dashrightarrow \PP^4$ (resp. $\eta_2:R_2(Y^m_5)\dashrightarrow \PP^4$, resp. $\eta_3:R_3(Y^m_5)\dashrightarrow Gr(2,5)$). Analyzing the fibers of these maps, we obtain the desired rationality(\S \ref{sramo}).
As a consequence, we can only expect to find rational varieties from the birational geometry of rational curves in $Y^m_5$, unlike the case of cubic hypersurfaces (cf. \cite{BD, LLSV}).

After proving Theorem \ref{sI3}, we investigate compactified moduli spaces of rational curves in $Y^m_5$.
In \S\ref{sy65}, we describe the birational geometry of compactified moduli spaces of rational curves of degree $d\le 3$ in $Y^6_5=Gr(2,5)$ from \cite{CK,CHK}. For lines, it is straightforward that $R_1(Y^6_5)=H_1(Y^6_5)$ is the flag variety $Gr(1,3,5)$.
For conics, the quasimap perspective (cf. \cite{KM}) gives us  a compactification
\beq\label{sI2} \PP (\mathrm{Hom}(\CC^2,\CC^2)^{\oplus 5})/\!/SL_2\times SL_2.\eeq
The two $SL_2$ act on the two $\CC^2$ in the standard manner and hence we find that \eqref{sI2} is the quiver variety associated to the quiver with two vertices and five edges connecting the vertices (\cite[Proposition 15]{Dre87} and \cite{Kin94}).

In \S\ref{sy55}, we prove that
\beq\label{sI4} H_1(Y^5_5)=R_1(Y^5_5)=\text{blow-up of }Gr(3,5)\text{ along a smooth quadric 3-fold }\Sigma.\eeq
We also prove that the Fano variety $F_2(Y^5_5)$ of planes in $Y^5_5$ is the disjoint union of the blowup of $\PP^4$ at a point and the smooth quadric 3-fold $\Sigma$ (cf. \cite[\S 3.2]{iliev3}).
We show that $H_2(Y^5_5)$ can be obtained from a $Gr(3,5)$-bundle over $\PP^4$ by a single blow-up/-down (Proposition \ref{birmodel}).

In \S\ref{sy45}, we recall classical results of Todd (cf. \cite{Todd}) on the Fano variety $F_2(Y^4_5)$ of planes and the moduli space $R_1(Y^4_5)=H_1(Y^4_5)$ of lines in $Y^4_5$. We give elementary proofs of these results. 
By the same method in the $5$-fold case, we further prove that $H_2(Y^4_5)$ can be obtained from a $\PP^3$-bundle over $\PP^4$ by a single blow-up/-down (Proposition \ref{birmodel2}).
As a direct corollary, 
\begin{theo}(Proposition \ref{birmodel}, \ref{birmodel2})
Hilbert schemes $H_2(Y^m_5)$ for $m=4, 5$ of conics are irreducible and smooth.
\end{theo}

The method studying about lines or conics in this paper can be applied in the $3$-fold case $Y_5^3$. In \S \ref{sy35}, we reprove the well-known results in $\eqref{sI1}$ about the space of lines and conics in $Y_5^3$.

\subsection{Notations}
All the schemes in this paper are defined over $\CC$ and $Gr(k,n)$ denotes the space of $k$-dimensional subspaces in $\CC^n$. Let $\{e_0,e_1,\cdots, e_{n-1}\}$ be the standard basis of $\CC^n$ unless otherwise stated. Let $p_{i_1i_2 \cdots i_k}$'s denote the Pl\"ucker coordinates of the Grassmannian $Gr(k,n)\hookrightarrow \PP(\wedge^k \CC^n)$.
\subsection*{Acknowledgements}
The first named author wishes to express his gratitude to Atanas Iliev and Young-Hoon Kiem for leading the interest in this topic. We would like to thank Han-Bom Moon and Wanseok Lee for valuable discussions and comments.
We also thank the anonymous reviewer for valuable comments and suggestions to improve the quality of the paper.

\section{Rational curves in Grassmannians}\label{sG}
In this section, we  classify rational curves of degree $\le 3$ in Grassmannians. Let $Gr(2,n), (n\ge 4)$ be the Grassmannian of lines in $\PP^{n-1}$ and consider the Pl\"ucker embedding
$$G:=Gr(2,n)\hookrightarrow \PP(\wedge^2 \CC^n)= \PP^{\binom{n}{2}-1}.$$
For fixed subspaces $V_1\subset V_2 \subset \CC^n$. Let $\sigma_{a_1,a_2}=\{[L]\in G | \mathrm{dim} (L\cap V_i)\geq i\}$ be the Schubert variety (class) where $a_i:=n-2+i-\mathrm{dim}(V_i)$ for $i=1,2$.
\def\La{\Lambda }
\def\taula{\tau_{\La} }
\def\talai{\tau_{\Lambda_i} }
\def\taula1{\tau_{\Lambda_1} }
\def\taula2{\tau_{\Lambda_2} }
\def\taula3{\tau_{\Lambda_3} }
\def\silap{\sigma_{3,2}(p,\Lambda) }
\def\deg{\mathrm{deg}\, }

To study the rational curves in $Gr(2,n)$, we only consider the following collection of Schubert cycles in $Gr(2,n)$.
\begin{defi}\label{H4} Let us think of a point $\ell\in Gr(2,n)$ as a line in $\PP^{n-1}$ and fix a flag $p\in \PP^1\subset \PP^2\subset \PP^3\subset \PP^{n-1}$.
Then we define
\begin{itemize}
\item $\sigma_{n-4,0}=\{\ell\,|\, \ell\cap \PP^2\ne \emptyset\} \quad \quad (\dim n, \deg  n(n-3)/2)$
\item $\sigma_{n-3,0}=\{\ell\,|\, \ell\cap \PP^1\ne \emptyset\} \quad \quad (\dim n-1, \deg  n-2)$
\item $\sigma_{n-4,n-4}=\{\ell\,|\, \ell\subset \PP^3\} \quad \quad \quad(\dim 4, \deg  2)$
\item $\sigma_{n-3,n-4}=\{\ell\,|\, \ell\cap \PP^1\ne \emptyset, \ell\subset \PP^3 \} \quad (\dim 3, \deg  2)$
\item $\sigma_{n-2,0}=\{\ell\,|\, p\in\ell\} \quad \quad \quad \quad (\dim n-2, \deg  1)$
\item $\sigma_{n-3,n-3}=\{\ell\,|\, \ell\subset \PP^2\} \quad \quad \quad (\dim 2, \deg  1)$
\item $\sigma_{n-2,n-4}=\{\ell\,|\, p\in \ell\subset \PP^3\} \quad \quad (\dim 2, \deg  1)$
\item $\sigma_{n-2,n-3}=\{\ell\,|\, p\in\ell\subset \PP^2\} \quad \quad (\dim 1, \deg  1)$.
\end{itemize}\end{defi}
The dimensions (and the degrees) of the Schubert cycles comes from \cite[Page 196]{GrHa} and \cite[Example 14.7.11]{Ful}. When $n=5$, these cycles freely generate the homology $H_*(Gr(2,5),\ZZ)$. 
Let $S(C,C')$ be the \emph{rational normal scroll} arising from the rational normal curves $C$ and $C'$ (allowing $C$ to be a point).
\begin{prop}\label{G1} Let $C$ be a smooth rational curve in $G=Gr(2,n)$ of degree $d$ with respect to the Pl\"ucker embedding.
\begin{enumerate}
\item If $d=1$, $C$ is the Schubert variety $\sigma_{n-2,n-3}(p,\Lambda)$ of lines in a fixed plane $\Lambda\subset \PP^{n-1}$ passing through a fixed point $p\in \Lambda$.
\item If $d=2$, $C$ is either the variety of lines of the ruling of the scroll $S(p, C')$ for a point $p$ and a conic $C'$ in $\PP^3$, 
or the variety of lines of ruling of the scroll $S(\ell_0, \ell_1)$ for two lines $\ell_0$ and $\ell_1$.
\item If $d=3$, $C$ is either the variety of lines of the ruling of $S(p, C')$ for a point $p$ and a twisted cubic curve $C'$ in $\PP^{n-1}$, or the variety of lines of the ruling of $S(\ell, C'')$ for a line $\ell$ and a conic $C''$.
\end{enumerate}
\end{prop}

\begin{proof}
If $\ell$ is a line in $\PP^{n-1}$, the locus of hyperplanes in $\PP^{n-1}$ containing $\ell$ is isomorphic to $\PP^{n-3}\subset \PP^{(n-1)*}$. Thus the locus of hyperplanes in $\PP^{n-1}$ containing a line in the family parameterized by the curve $C$ has dimension at most $n-2<n-1=\mathrm{dim}\PP^{(n-1)*}$. Let us choose a point $[\Pi]\in \PP^{(n-1)*}$ of the complement of this locus.
Then $\Pi \subset \PP^{n-1}$ meets each line in the family parameterized by $C$ transversely (\cite[Chapter I, Theorem 7.1]{Hart}). 
Let $C\leftarrow F\stackrel{f}{\rightarrow} \PP^{n-1}$ denote the family of lines parameterized by $C$ where $\pi:F\to C$ is a $\PP^1$-bundle and  $f|_{\pi^{-1}(x)}$ is the embedding of the line represented by $x$ so that we have a Cartesian diagram
\[\xymatrix{
f^{-1}(\Pi)\ar[r]\ar[d] & \Pi\ar[d]\\
F\ar[d]\ar[r]^f & \PP^{n-1}\\
C}\]
Since $\pi^{-1}(x)$ intersects with $\Pi$ transversely, locally we can write the bijective map $f^{-1}(\Pi)\to F\to C$ as
$$\{(z_1,z_2)\,|\,z_2=g(z_1)\}\subset \CC^2 \to \CC, \quad (z_1,z_2)\mapsto z_1$$
and find that $C_0:=f^{-1}(\Pi)$ is the image of a section $s_0:C\to F$. Let $N_{C_0/F}$ be the normal bundle to $C_0$ in $F$. Then $F=\PP(\cO_C\oplus N)$ where $N=s_0^*N_{C_0/F}$. Let $s_1:C\cong \PP N \hookrightarrow \PP(\cO_{C}\oplus N)=F$. Let $L_0=(f\circ s_0)^*\cO_{\PP^{n-1}}(1)$ and $L_1=(f\circ s_1)^*\cO_{\PP^{n-1}}(1)$ so that the morphism $f\circ s_0:C\to \PP^{n-1}$ is $(a_0:a_1:\cdots:a_{n-1})$ for $a_i\in H^0(C,L_0)$ while $f\circ s_1:C\to \PP^{n-1}$ is $(b_0:b_1:\cdots:b_{n-1})$ for $b_i\in H^0(C,L_1)$.
Then the family $F$ of lines can be represented by two dimensional subspaces of $\CC^n$ spanned by the rows of
\[\left(\begin{matrix}
a_0&a_1&a_2&a_3&\cdots&a_{n-1}\\
b_0&b_1&b_2&b_3&\cdots&b_{n-1}
\end{matrix} \right)\]
and hence the Pl\"ucker coordinates for $C\subset Gr(2,n)\subset \PP^{\binom{n}{2}-1}$ are $a_ib_j-a_jb_i\in H^0(C,L_0\otimes L_1)$. Therefore, we find that the degree of $C$ is $$d=d_0+d_1,\quad \text{where } d_0=\deg L_0\text{ and } d_1=\deg L_1.$$

If $d=1$, then $C\cong \PP^1$ and, without loss of generality, $L_0=\cO_{\PP^1}$ and $L_1=\cO_{\PP^1}(1)$.  If we let $p=(a_0:a_1:\cdots:a_{n-1})$ and $\La$ be the plane spanned by $p$, $(b_0(0):\cdots:b_{n-1}(0))$ and $(b_0(1):\cdots:b_{n-1}(1))$, we obtain (1)

Let $d=2$ so that $C\cong \PP^1$. If $d_0=0$ and $d_1=2$, letting $p=(a_0:a_1:\cdots:a_{n-1})$ and $C'=\{(b_0(t):\cdots:b_{n-1}(t))\,|\,t\in \PP^1\}$, we find that $C$ is the variety of lines passing through $p$ and points in the conic $C'$. If $d_0=1$ and $d_1=1$, both the images $\ell_0:=f\circ s_0(C)$ and $\ell_1:=f\circ s_1(C)$ are lines and $C$ parameterizes lines through a point moving in $\ell_0$ and another point moving in $\ell_1$.

The case for $d=3$ is similar.
\end{proof}
\begin{rema}
We could also prove Proposition \ref{G1} by using Grothendieck's theorem which says that any vector bundle on $\PP^1$ splits into the direct sum of line bundles. If $f:\PP^1\to G=Gr(2,n)$ is a morphism of degree $d$, the pullback of the (dual) universal bundle gives a surjective homomorphism \[\varphi:\cO_{\PP^1}^{\oplus n}\to \cO_{\PP^1}(d_1)\oplus \cO_{\PP^1}(d_2)\]over $\PP^1$ with $d=d_1+d_2, d_1\le d_2$.

When $d=1$, we have $d_1=0$ and $d_2=1$. The composition $\pi_1\circ\varphi:\cO^{\oplus n}\to \cO$ of $\varphi$ with the projection onto the first factor is surjective and gives us the point $p\in \PP^{n-1}$ while that with the second projection $\pi_2\circ \varphi$ gives a line in $\PP^{n-1}$. We thus obtain Proposition \ref{G1} (1).

When $d=2$, we have $(d_1,d_2)=(0,2)$ or $(1,1)$. In the first case, $\pi_1\circ\varphi$ gives a  point $p$ while $\pi_2\circ\varphi$ gives a conic $C'$. In the second case, both $\pi_1\circ\varphi$ and $\pi_2\circ\varphi$ give us lines.

When $d=3$, we have $(d_1,d_2)=(0,3)$ or $(1,2)$. In the first case, $\pi_1\circ\varphi$ gives a  point $p$ while $\pi_2\circ\varphi$ gives a twisted cubic $C'$. In the second case,  $\pi_1\circ\varphi$ gives a  line  $\ell$ while $\pi_2\circ\varphi$ gives a conic $C'$.
\end{rema}
As a consequence of Proposition \ref{G1}, we obtain the following characterization.
\begin{prop}\label{H1}
(1) (\cite[Exercise 6.9]{Harris}) The variety $R_1(G)$ of lines in $G=Gr(2,n)$ is the flag variety $Gr(1,3,n)$ of $V_1\subset V_3\subset \CC^n$ with $\dim V_i=i$.

(2) For any smooth conic $C\subset G\subset \PP^{\binom{n}{2}-1}$, there is a unique three dimensional subspace $\PP^3\subset \PP^{\binom{n}{2}-1}$ that contains all the lines in $\PP^{n-1}$ parameterized by $C$.

(3) For a general twisted cubic $C\subset G\subset \PP^{\binom{n}{2}-1}$, there is a line $\ell\subset \PP^{n-1}$ which meets all the lines in $\PP^{n-1}$ parameterized by $C$.
\end{prop}
\begin{proof}
By Proposition \ref{G1} (1), each line in $G$ parameterizes the family of lines in a plane $\PP^2\subset \PP^{n-1}$ passing through a point $p\in \PP^2$. Conversely, such a family of lines in $\PP^{n-1}$ gives a line in $G$.

By Proposition \ref{G1} (2), each conic $C$ in $G$ parameterizes lines joining a point $p$ and points in a conic $C'$, or points in a line $\ell_0$ and another line $\ell_1$. If we choose a $\PP^3$ containing $p$ and $C'$ or $\ell_0$ and $\ell_1$, then all the lines in the family parameterized by $C$ are contained in the $\PP^3$.

By Proposition \ref{G1} (3), each twisted cubic $C$ in $G$ parameterizes lines joining a point $p$ and points in a twisted cubic $C'$ or points in a line $\ell$ and a conic $C''$. If we choose any line through $p$ meeting with the cubic $C'$ in the first case or the line $\ell$ in the second case, we find that all the lines in the family parameterized by $C$ meet $\ell$.
\end{proof}

\begin{defi}\label{H3}
(1) We call the point $p$ in Proposition \ref{G1} (1), the {\emph{vertex}} of the family of lines parameterized by the line.

(2) We call the three dimensional subspace $\PP^3$ in Proposition \ref{H1} (2), the {\emph{envelope}} of the family of lines parametrized by the conic $C$.

(3) We call the line $\ell$ in Proposition \ref{H1} (3), the {\emph{axis}} of the family of lines parametrized by the twisted cubic $C$.
\end{defi}

\begin{coro}\label{H2}(cf. \cite[Exercise 6.9]{Harris} and \cite{CC10})
Let $R_d(G)$ denote the moduli space of smooth rational curves of degree $d\le 3$ in $G=Gr(2,n)$ with $n\ge 4$.

(1) There is a morphism $\eta_1:R_1(G)\to \PP^{n-1}$ which assigns the vertex to each line in $G$. Each fiber is isomorphic to $Gr(2,n-1)$.

(2) There is a rational map $\eta_2:R_2(G)\dashrightarrow Gr(4,n)$ which assigns the envelope to each smooth conic in $G$. A general fiber is isomorphic to the moduli space $R_2(\sigma_{n-4,n-4})$ of conics in the Schubert cycle $\sigma_{n-4,n-4}(\cong Gr(2,4))$ in Definition \ref {H4}.

(3) There is a rational map $\eta_3:R_3(G)\dashrightarrow Gr(2,n)$ which assigns the axis to each twisted cubic curve in $G$. A general fiber is isomorphic to the moduli space $R_3(\sigma_{n-3,0})$ of smooth twisted cubics in the Schubert cycle $\sigma_{n-3,0}$ in Definition \ref{H4}.
\end{coro}
\begin{proof}
(1) By Proposition \ref{H1} (1), $\eta_1$ is the forgetful map $R_1(G)=Gr(1,3,n)\to Gr(1,n)$ defined by $(V_1,V_3)\to V_1$. The choice of $V_3/V_1$ in $\CC^n/V_1$ is parameterized by $Gr(2,n-1)$.

(2) General conic $C$ is given by $(p,C')$ or $(\ell_0,\ell_1)$ in the notation of Proposition \ref{G1} (2)  which span a unique $\PP^3\subset \PP^{n-1}$. In this case $C\in R_2(\sigma_{n-4,n-4})$.

(3) General twisted cubic $C$ is given by $(\ell, C'')$ in the notation of Proposition \ref{G1} (3). Lines meeting $\ell$ form the Schubert cycle $\sigma_{n-3,0}$.
\end{proof}

\section{Rationality of moduli spaces of rational curves in linear sections of Grassmannian}\label{sramo}

Before embarking on the study of the birational geometry of moduli spaces of rational curves in Grassmannians, we determine their birational types.
The purpose of this section is to prove the following rationality result.

\begin{theo}\label{H14} Let $G=Gr(2,5)$. Let $Y=Y_5^m$ be the intersection of $Gr(2,5)\subset \PP^{\binom{5}{2}-1}$ with $6-m$ general hyperplanes. Then the moduli spaces $R_d(Y)$ of smooth rational curves of degree $d$ on $Y$ are all rational for $d\le 3$.
\end{theo}

When $m=0$, $Y^0_5$ consists of five points since the degree of $G$ is $5$. When $m=1$, $Y^1_5$ is a smooth elliptic curve of degree $5$ and hence there is no rational curve in $Y^1_5$.
\begin{lemm}\label{H5}
\begin{enumerate}
\item $R_1(Y^2_5)$ consists of 10 reduced points;
\item $R_2(Y^2_5)$ is the disjoint union of five copies of $\PP^1-\{0,1,\infty\}$;
\item $R_3(Y^2_5)$ is the disjoint union of four copies of $\PP^2-\PP^1$ and $\PP^2$ minus four lines;
\item there are no planes in $Y^2_5$.
\end{enumerate}
\end{lemm}
\begin{proof}
$Y^2_5$ is a del Pezzo surface of degree five and hence isomorphic to the blowup of $\PP^2$ at four general points. Obviously $Y^2_5$ does not contain any plane.

By adjunction, a line in $Y^2_5$ is a rational curve in $Y^2_5$ with self-intersection number $-1$. There are four exceptional curves and the (proper transforms of) six lines in $\PP^2$ passing through two out of the four blowup centers. Hence there are exactly ten lines in $Y^2_5$.

By adjunction again, a smooth conic in $Y^2_5$ is a rational curve with self-intersection number $0$. These are the (proper transforms of) lines through one of the four blowup centers or conics through the all four blowup centers, minus the lines through two out of the four blowup centers and the three degenerate conics through the four points.

By adjunction, a smooth twisted cubic in $Y^2_5$ is a rational curve with self-intersection number $1$. These are the (proper transforms of) lines in $\PP^2$ not passing through any of the four blow up centers or conics passing through three out of the four blowup centers. The first collection is $\PP^2$ minus four lines while the second is four copies of $\PP^2$ minus a line.
\end{proof}

\begin{prop}\label{H11}(\cite{Fae, FuNa, Ili, San})
The Hilbert schemes $H_d(Y^3_5)$ with Hilbert polynomial $dt+1$ in $Y^3_5$ are isomorphic to
\begin{equation}\label{delp1}H_1(Y^3_5)\cong\PP^2,\quad H_2(Y^3_5)\cong\PP^4,\quad \mathrm{and} \quad H_3(Y^3_5)\cong \mathrm{Gr}(2,5).\end{equation}

In particular, $R_d(Y^3_5)$ for $d\le 3$ are rational.
\end{prop}
For the completeness of our argument using in this paper, we deal with the same results about the space of line and conics in $Y_5^3$ (\S \ref{sy35}).
\begin{rema}\label{corr}
The correspondences in \eqref{delp1} are given by the composition map $\eta_d\circ i$ for $i:R_d(Y^3_5)\subset R_d(G)$ and the map $\eta_d$ in Corollary \ref{H2}. Let us provide a geometric description of the fiber of $\eta_d\circ i$ for $d=2,3$ (cf. \cite[\S 1]{Ame04} and \cite[Remark 2.47]{San}). Since $\sigma_{1,1} (\PP^3)\cong Gr(2,4)$ has degree two (Definition \ref{H1}), the intersection $\sigma_{1,1}(\PP^3) \cap H_1\cap H_2\cap H_3$ is a conic in $Y_5^3$. 

Similarly, by an explicit computation with coordinates, we find that $\sigma_{2,0}(\PP^1)$ is contained in a linear subspace $\PP^6$ of $\PP^9$ and $\sigma_{2,0}(\PP^1)$ is defined by three quadric polynomial equations. Thus the intersection $\sigma_{2,0}(\PP^1)$ with $H_1\cap H_2\cap H_3$ is a twisted cubic, which is the fiber of the map $\eta_3\circ i$ (cf. \cite[Proposition 4.5]{IKKR16}).
\end{rema}

\begin{coro}\label{H12}
$R_d(Y^m_5)$ are irreducible for $m\ge 3$ and $d\le 3$.
\end{coro}
\begin{proof}
Note that the spaces $Y^m_5$ are isomorphic to each other with respect to the choice of the hyperplane sections $H_i$. Let \[I= \{(C, H)\in R_d(Y_5^m)\times Gr(13-m,10)| C\subset H\}\]
be the incident variety of the pair of the curve $C$ and the subspace $H\subset \PP^9$ whose codimension is $m-3$.
The second projection map $p_2:I \to Gr(13-m,10)$ is clearly dominant. Moreover, $Gr(13-m,10)$ is irreducible. Since $p_2^{-1}(H)=R_d(Y_5^3)$ is irreducible for the general $H$ (Proposition \ref{H11}), the incident variety $I$ is irreducible. Now the first projection map $p_1: I \lr R_d(Y_5^m)$ is dominant because each smooth rational curve of degree $d\le 3$ in $Y^m_5\subset Gr(2,5)\subset \PP^9$ is contained in a $\PP^3$ and hence in a 
subspace $H\subset \PP^9$ of codimension $m-3$. This proves the claim.
\end{proof}
With the irreducibility (Corollary \ref{H12}) at hand, the rationality of $R_d(Y^m_5)$ for $m\ge 4$ and $d\le 3$ is obtained through the following lemmas.
\begin{lemm}(cf. \cite[Theorem 3]{KP01} and \cite[Theorem 4.9]{Landsberg})\label{H6}
Let $Y^6_5=Gr(2,5)$. Then
\begin{enumerate}
\item $R_1(Y^6_5)=Gr(1,3,5)$ is a $Gr(2,4)$-bundle over $\PP^4$;
\item $R_2(Y^6_5)$ is birational to a $Gr(3,6)$-bundle over $Gr(4,5)=\PP^4$;
\item $R_3(Y^6_5)$ is birational to a $Gr(4,7)$-bundle over $Gr(2,5)$;
\item the Fano variety of planes in $Y^6_5$ is $Gr(1,4,5)\sqcup Gr(3,5)$.
\end{enumerate}
\end{lemm}
\begin{proof}
By Proposition \ref{G1}, giving a line in $Y^6_5$ is equivalent to giving a pair of a point $p\in \PP^4$ and a plane $\PP^2\subset \PP^4$ containing $p$ (vertex). So we get (1).

By Corollary \ref{H2} (2), we have the rational map (envelope) $\eta_2:R_2(Y^6_5)\dashrightarrow \PP^4$ whose general fiber is $R_2(\sigma_{1,1})$. By Pl\"ucker embedding, $\sigma_{1,1}\cong Gr(2,4)\subset \PP^5$ is a quadric hypersurface. Hence a general $\PP^2$ in $\PP^5$ gives a smooth conic $Gr(2,4)\cap \PP^2$. Therefore if we let $U\to Gr(4,5)\cong \PP^4$ be the universal rank 4 bundle, a general point in the relative Grassmannian $Gr(3, \wedge^2 U)$ over $\PP^4$ gives a conic in $Gr(2,5)$. This clearly gives the birational map (2).

By Corollary \ref{H2} (3), we have the rational map (axis) $\eta_3:R_3(Y^6_5)\dashrightarrow Gr(2,5)$ whose general fiber is $R_3(\sigma_{2,0})$. By direct computation with coordinates, we find that $\sigma_{2,0}$ is contained in a linear subspace $\PP^6$ of $\PP^9$ and $\sigma_{2,0}$ is defined by three quadric polynomial equations. As mentioned in Remark \ref{corr}, a general $\PP^3$ in $\PP^6$ intersects with $\sigma_{2,0}$ along a twisted cubic. 
Hence let $\cQ$ be the universal quotient bundle on $Gr(2,5)$. Let $\cK_7$ be the kernel of surjection $\wedge^2 \cO^{\oplus 5}\twoheadrightarrow \wedge^2 \cQ$. Then a general point in the relative Grassmannian $Gr(4,\cK_7)$ over $Gr(2,5)$ gives a twisted cubic in $Gr(2,5)$. This proves item (3).

A plane in $Gr(2,5)$ is either the collection of lines in $\PP^3\subset \PP^4$ passing through a point $p\in \PP^3$ or the collection of lines in a plane $\PP^2\subset \PP^4$. The first is parameterized by $Gr(1,4,5)$ while the second is by $Gr(3,5)$.
\end{proof}

\begin{lemm}\label{H7}
\begin{enumerate}
\item $R_1(Y^5_5)$ is birational to a $Gr(2,3)\cong\PP^2$-bundle over $\PP^4$;
\item $R_2(Y^5_5)$ is birational to a $Gr(3,5)$-bundle over $Gr(4,5)\cong\PP^4$;
\item $R_3(Y^5_5)$ is birational to a $Gr(4,6)$-bundle over $Gr(2,5)$.
\end{enumerate}
\end{lemm}
\begin{proof}
Let $Y^5_5=Gr(2,5)\cap H_1\subset \PP^9$ where $H_1$ is a hyperplane.
The inclusion $Y^5_5\subset Y^6_5=Gr(2,5)$ induces the inclusion $$\imath_d:R_d(Y^5_5)\hookrightarrow R_d(Y^6_5).$$

For $p\in \PP^4$,
the Schubert cycle $\sigma_{3,0}(p)=\{\ell\in Gr(2,5)=Y^6_5\,|\, p\in \ell\}\cong \PP^3$ embeds into $\PP^9$ as a linear subvariety. By Bertini's theorem, for general $p$, $\sigma_{3,0}(p)$ intersects with $H_1$ cleanly along a $\PP^2$.
Therefore
a general fiber of $\eta_1\circ\imath_1:R_1(Y^5_5)\to \PP^4$ is $R_1(H_1\cap \sigma_{3,0}(p))\cong Gr(2,3)\cong \PP^2$.

A general fiber of $\eta_2\circ\imath_2:R_2(Y^5_5)\dashrightarrow Gr(4,5)\cong \PP^4$ is $R_2(\sigma_{1,1}\cap H_1)$ for a hyperplane $H_1$ of $\sigma_{1,1}=Gr(2,4)\subset \PP^5$. Since $\sigma_{1,1}$ is a quadric hypersurface, this fiber is birational to $Gr(3,5)$ which parameterizes $\PP^2$'s in $H_1\cap \PP^5$. Hence let $\cU$ be the universal subbundle on $Gr(4,5)$. Let 
$\cK_5:= \mathrm{ker}\{\wedge^2\cU\subset \wedge^2 \cO^{\oplus 5}\to \cO \}$
be the kernel of the composition map where the second arrow is given by the hyperplane $H_1$. Then the general point in the relative Grassimannian $Gr(3,\cK_5)$ over $Gr(4,5)$ gives a conic in $Y_5^5$.

A general fiber of $\eta_3\circ\imath_3:R_3(Y^5_5)\dashrightarrow Gr(2,5)$ is $R_3(\sigma_{2,0}\cap H_1)$ for a hyperplane $H_1$ of $\PP^6$ in the notation of the proof of Lemma \ref{H6}.
By the proof of Lemma \ref{H6}, a general choice of $\PP^3$ in $H_1\cap \PP^6$ gives a twisted cubic $\sigma_{2,0}\cap \PP^3$ and hence the fiber above is birational to $Gr(4,6)$. Therefore let $\cK_7$ be the bundle of rank $7$ defined in item (3) of Lemma \ref{H6}. Let $\cK_6$ be the kernel of the composition map: $\cK_7\hookrightarrow \wedge^2 \cO^{\oplus 5}\to \cO$ where the second map is defined by $H_1$. Then the relative Grassmannian $Gr(4,\cK_6)$ over $Gr(2,5)$ provides the birational model.
\end{proof}

\begin{lemm}\label{H8}
\begin{enumerate}
\item $R_1(Y^4_5)$ is birational to $\PP^4$;
\item $R_2(Y^4_5)$ is birational to a $Gr(3,4)=\PP^3$-bundle over $Gr(4,5)=\PP^4$;
\item $R_3(Y^4_5)$ is birational to a $Gr(4,5)=\PP^4$-bundle over $Gr(2,5)$.
\end{enumerate}
\end{lemm}
\begin{proof}
The proof is identical to that for Lemma \ref{H7}, if we simply use $H_1\cap H_2$ in place of $H_1$ and replace $Y^5_5$ by $Y^4_5$ where $H_1$ and $H_2$ are hyperplanes.
\end{proof}
Combining the above lemmas and Corollary \ref{H12}, we obtain Theorem \ref{H14} because the Grassmannian variety $Gr(k,n)$ is rational.

In the remaining sections, we will provide a more detailed description of compactifications of $R_d(Y^m_5)$.

\section{Fano 6-fold $Y^6_5=Gr(2,5)$}\label{sy65}

Throughout this section, we let $G=Y^6_5=Gr(2,5)$
and consider the compactified moduli spaces of rational curves of degree $d\le 3$ in $G$.

Let us recall several compactifications of $R_d(G)$ that we studied in \cite{CHK}.

\medskip
\noindent$\bullet$ \textbf{Hilbert compactification}: Since $Y\subset \PP^9$, Grothendieck's general
construction gives us the Hilbert scheme $Hilb^{dt+1}(G)$ of
closed subschemes of $G$ with Hilbert polynomial $h(t)=dt+1$ as a
closed subscheme of $Hilb^{dt+1}(\PP^r)$. The closure
$H_d(G)$ of $R_d(G)$ in $Hilb^{dt+1}(G)$ is a compactification which we call the
\emph{Hilbert compactification}..

\medskip
\noindent $\bullet$ \textbf{Kontsevich compactification}:
A stable map is a morphism of a connected nodal curve $f:C\to G$
with finite automorphism group. Here two maps $f:C\to G$
and $f':C'\to G$ are isomorphic if there exists an isomorphism
$\eta:C\to C'$ satisfying $f'\circ \eta=f$. The moduli space $\cM_0(G,d)$
of isomorphism classes of stable maps $f:C\to G$ with arithmetic genus $0$ and
$\mathrm{deg}(f^*\cO_G(1))=d$ has a projective coarse moduli space.
The closure $M_d(G)$ of
$R_d(G)$ in $\cM_0(G,d)$ is a
compactification, called the \emph{Kontsevich
compactification}.

\medskip
\noindent $\bullet$ \textbf{Simpson compactification}: A coherent sheaf $E$ on
$G$ is \emph{pure} if any nonzero subsheaf of $E$ has the same
dimensional support as $E$. A pure sheaf $E$ is called
\emph{semistable} if \[ \frac{\chi(E(t))}{r(E)}\le
\frac{\chi(E''(t))}{r(E'')}\qquad \text{for }t>\!>0
\]
for any nontrivial pure quotient sheaf $E''$ of the same
dimension, where $r(E)$ denotes the leading coefficient of the
Hilbert polynomial $\chi(E(t))=\chi(E\otimes \cO_G(t))$. We obtain
\emph{stability} if $\le $ is replaced by $<$. If we replace the
quotient sheaves $E''$ by subsheaves $E'$ and reverse the
inequality, we obtain an equivalent definition of (semi)stability.
There is a projective moduli scheme $\cS
imp^{P}(G)$ of semistable sheaves on $G$ of a given Hilbert polynomial $P$.
If $C$ is a smooth rational curve in $G$, then the structure sheaf
$\cO_C$ is a stable sheaf on $G$. The closure $P_d(G)$ of $R_d(G)$ in $\cS imp^{dt+1}(G)$ is a compactification called the
\emph{Simpson compactification}.

\medskip

$R_1(G)=Gr(1,3,5)$ is compact and thus $H_1(G)=M_1(G)=P_1(G)=R_1(G)=Gr(1,3,5)$.
\subsection{Conics in $G=Gr(2,5)$}
Since $G$ is a homogeneous variety, we can apply the results of \cite{CHK}.
\begin{theo}\label{th1.1}\cite[\S3]{CHK} \begin{enumerate}
\item $P_2(G)\cong H_2(G)$.
\item The blow-up of $M_2(G)$ along
the locus of stable maps with linear image
is isomorphic to the smooth blow-up of $P_2(G)$ along the locus of stable sheaves with linear support.
\end{enumerate}
\end{theo}
For later use (\S \ref{birgeo}), let us discuss a birational geometry of the Hilbert scheme $H_2(G)$ which is related with the enveloping map (Corollary \ref{H2} (2)) \[\eta_2:R_2(G)\dashrightarrow \PP^4=Gr(4,5).\] Let $Gr(2,\cU)$ be the Grassmannian bundle over the universal bundle $\cU$ of $Gr(4,5)$. By \cite[Theorem 1.4]{KOL}, there exists a relative Hilbert scheme of conics
\[
\widetilde{\eta}_2: H_2(Gr(2,\cU))\lr Gr(4,5)
\]
with the canonical projection map $\widetilde{\eta}_2$. The Fano variety $F_2(Gr(2,n))$ of planes in $Gr(2,n)$ was described in \cite[Theorem 4.9]{Landsberg}.
In particular, $F_2(Gr(2,4))=Gr(1,4)\sqcup Gr(3,4)$ and $F_2(Gr(2,5))=Gr(1,4,5)\sqcup Gr(3,5)$. The planes parameterized by $Gr(1,4,5)$ (resp. $Gr(3,5)$) is called by $\sigma_{3,1}$ (resp. $\sigma_{2,2}$)-plane.
\begin{prop}\cite[\S 3.1]{iliev3}\label{bicom}
Under the above definition and notation, the natural birational morphism
\[
\Psi:H_2(Gr(2,\cU))\longrightarrow H_2(G)
\]
is a \emph{smooth} blow up morphism along the space $\Delta$ of conics lying in the $\sigma_{2,2}$-planes.
\end{prop}
\begin{proof}
The exceptional divisor is a $\PP^5$-bundle over a $Gr(3,4)$-bundle over $Gr(4,5)$. By its construction, the $Gr(3,4)$-bundle over $Gr(4,5)$ is isomorphic to the $Gr(1,2)=\PP^1$-bundle over $Gr(3,5)$ where $Gr(1,2)$ parameterizes $\PP^3$ containing a fixed $\PP^2$ in $\PP^4$. Hence, to show the smoothness of the blow-up $\Psi$, let us describe the normal space of $\Delta$ in $H_2(G)$ at $C$. From the normal bundle sequence $\ses{N_{C/\PP^2}}{N_{C/G}}{N_{\PP^2/G}|_{C}}$ and the short exact sequence $\ses{N_{\PP^2/G}(-2)}{N_{\PP^2/G}}{N_{\PP^2/G}|_{C}}$, we know that the normal space is isomorphic to $$N_{\Delta/H_2(G),C}\cong H^1(N_{\PP^2/G}(-2)).$$
From a diagram chasing, $N_{\PP^2/G}\cong \cQ\otimes \cO_{\PP^2}^{\oplus2}$ for the $\sigma_{2,2}$-type plane $\PP^2$.
Here $\cQ$ is the universal quotient bundle on $\PP^2$. This implies that the later space is $H^1(N_{\PP^2/G}(-2))\cong H^0(\cO_{\PP^2}^{\oplus2})^*$. This space is naturally identified with the choice of $\PP^3$ in $\PP^4$ while containing the plane $\PP^2$.
\end{proof}

\begin{prop}\label{chen}
Let $S(G)= Gr(3,\wedge^2\cU)$ be the relative Grassimannian bundle of the universal bundle $\cU$ over $Gr(4,5)$. Then there exists a smooth blow-up morphism
\[
\Xi:H_2(Gr(2,\cU))\longrightarrow S(G)
\]
where the blow up center is the disjoint union of the flag varieties $T(G):=Gr(1,4,5)\sqcup Gr(3,4,5)$.
\end{prop}
\begin{proof}
By the base change property of the blow-up, it is enough to check the claim fiberwisely. The later one has been done in \cite[Lemma 3.9]{CC10}. Note that $T(G)$ is isomorphic to the relative orthogonal bundle $OG(3,\wedge^2\cU)$ over $Gr(4,5)$ (\cite[Proposition 4.16]{HT15}).
\end{proof}
By Proposition \ref{bicom} and \ref{chen}, we obtain the following diagram:
\begin{equation}\label{diahil}
\xymatrix{
&&H_2(Gr(2,\cU))\ar[dll]_{\Xi}\ar[drr]^{\Psi}\ar@{=>}^{\widetilde{\eta_2}}[dd]&&\\
S(Gr(2,5))\ar@{=>}[drr]&&&&H_2(Gr(2,5))\ar@{-->}_{\eta_2}[dll]\\
&&Gr(4,5),&&
}
\end{equation}
where $\cU$ is the universal subbundle over $Gr(4,5)$.

Let us finish this subsection after some remarks about stable map space similar to the diagram \eqref{diahil} from the viewpoint of the birational geometry (\cite{CC10}). Since we do not use this part again in the remaining of the paper, we only sketch the results. Let $M_2(Gr(2,\cU))$ be the space of relative stable maps of degree two and genus zero. Let $M_2(Gr(2,\cU))\lr M_2(Gr(2,5))$ be the forgetful map. Let $N(Gr(2,5))$ be the relative Kronecker quiver space $N(\cU;2,2)$ with the fiber $N(4;2,2)$ (for the precise definition of quiver space, see \cite{CM15}). Then there exists a divisorial contraction $M_2(Gr(2,\cU))\lr N(Gr(2,5))$ contracting the stable maps whose images are planar (\cite{CM15}). In summary, we obtain
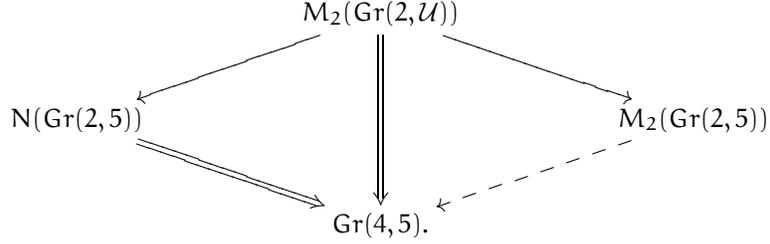
\[
\xymatrix{
&&M_2(Gr(2,\cU))\ar[dll]\ar[drr]\ar@{=>}[dd]&&\\
N(Gr(2,5))\ar@{=>}[drr]&&&&M_2(Gr(2,5))\ar@{-->}[dll]\\
&&Gr(4,5).&&
}
\]
Consider the universal rank 2 subbundle $U\hookrightarrow \cO_Y^{\oplus 5}$ and its dual surjection $\cO_Y^{\oplus 5}\to U^\vee$. Given a general conic $\PP^1\cong C\hookrightarrow Gr(2,5)$, the restriction of $U^\vee$ splits as $U^\vee|_{C}\cong \cO_{\PP^1}(1)\oplus \cO_{\PP^1}(1)$ and the universal quotient map is $\cO_{\PP^1}^{\oplus 5}\to \cO_{\PP^1}(1)\oplus \cO_{\PP^1}(1)$. Hence general rational curves are parameterized by an open subset of the geometric invariant theory quotient
$$\PP (H^0(\PP^1,\cO(1))\otimes \CC^2\otimes \CC^5)/\!/SL_2\times SL_2$$
where the first $SL_2$ acts on $\PP^1$ in the standard manner while the second acts on $\CC^2$ by matrix multiplication.
This GIT quotient is obviously the quiver variety $N(5;2,2)$ associated to the quiver with two vertices decorated with 2-dimensional vector spaces and five edges connecting the vertices. The geometry of the moduli space of stable maps $M_2(Gr(2,5))$ in the view point of Mori program has been studied in \cite{CM16b}.

\subsection{Twisted cubics in $G=Gr(2,5)$}
Because $G$ is homogeneous, we can apply the results of \cite{CHK} again.
\def\Ext{\mathrm{Ext} }
\def\Hom{\mathrm{Hom} }
\begin{theo}\label{th1.2} \cite[\S4]{CHK}
\begin{enumerate}
 \item $H_3(G)$ is the smooth blow-up of
  $P_3(G)$ along the locus $\Delta(G)$
of planar stable sheaves.
  \item  $P_3(G)$
  is obtained from $M_3(G)$ by three weighted blow-ups
  followed by three weighted blow-downs. In other words,
  $P_3(G)$ is obtained from $M_3(G)$ by blowing up along $\Gamma^1_0$,
$\Gamma^2_1$, $\Gamma^3_2$ and then blowing down along
$\Gamma^2_3$, $\Gamma^3_4$, $\Gamma^1_5$ where $\Gamma^j_i$ is the
proper transform of $\Gamma^j_{i-1}$ if $\Gamma^j_{i-1}$ is not
the blow-up/-down center and the image/preimage of
$\Gamma^j_{i-1}$ otherwise. Here $\Gamma_0^1$ is the locus of stable maps whose images are lines; $\Gamma_0^2$ is the locus of stable maps whose images consist of two lines; $\Gamma_1^3$ is the subvariety of the exceptional divisor $\Gamma_1^1$ which is a fiber bundle over $\Gamma^1_0$ with fibers
$$\PP \Hom_1 (\CC^2,\Ext^1_G(\cO_L,\cO_L(-1)))\cong \PP^1\times \PP \Ext^1_G(\cO_L,\cO_L(-1))$$
where $\Hom_1$ denotes the locus of rank 1 homomorphisms.
\end{enumerate}
\end{theo}
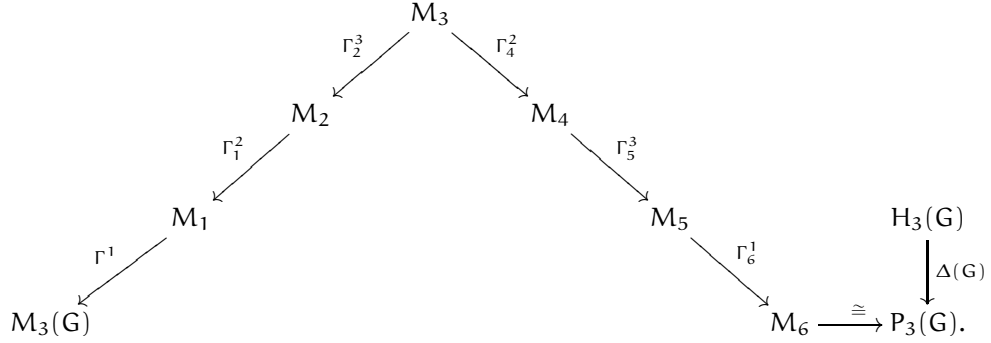
\[
\xymatrix{ &&&M_3\ar[dl]_{\Gamma_2^3}\ar[dr]^{\Gamma_4^2}\\
&&M_2\ar[dl]_{\Gamma_1^2}&&M_4\ar[dr]^{\Gamma^3_5}\\
&M_1\ar[dl]_{\Gamma^1}&&&&M_5\ar[dr]^{\Gamma^1_6}&&H_3(G)\ar[d]^{\Delta(G)}\\
M_3(G)&&&&&&M_6\ar[r]^\cong & P_3(G). }\]

\section{Fano 5-fold $Y^5_5$}\label{sy55}
In this section, we let $Y^5_5$ be the intersection of the Grassmannian $G=Gr(2,5)$ and a general hyperplane $H$.
For explicit calculation, we will let $$H=\{p_{12}-p_{03}=0\}$$
where $p_{ij}$'s denote the Pl\"ucker coordinates.

\subsection{Fano varieties of lines and planes in $Y_5^5$}
Let $Y:=Y_5^5=Gr(2,5)\cap H$. In this section, we give a precise description of the Fano variety $F_1(Y)$ of lines in $Y$ and the Fano variety $F_2(Y)$ of planes in $Y$ in the following two propositions.

\begin{prop}\label{sy551}
$F_1(Y)=H_1(Y)=S_1(Y)=M_1(Y)$ is the blow-up of the Grassmannian $Gr(3,5)$ along a smooth quadric 3-fold $\Sigma$.
\end{prop}
\begin{proof}
Each line $Z$ in $G$ is $\{\ell\in G\,|\, p\in \ell\subset P\}$ for a plane $P\subset \PP^4$ and a point $p\in P$.
We have a morphism $$\psi:F_1(Y)\hookrightarrow F_1(G)=Gr(1,3,5)\lra Gr(3,5),\quad (p,P)\mapsto P.$$

For a plane $P\in Gr(3,5)$ represented by
\[\left( \begin{matrix}
1&0&0&a_3&a_4\\
0&1&0&b_3&b_4\\
0&0&1&c_3&c_4
\end{matrix}\right),\]
let us consider lines in $P$ represented by
\[\left( \begin{matrix}
1&0&\alpha&a_3+\alpha c_3&a_4+\alpha c_4\\
0&1&\beta&b_3+\beta c_3&b_4+\beta c_4
\end{matrix}\right).\]
The equation $p_{12}=p_{03}$ gives a unique line $Z$ in $Y$ defined by $\alpha+b_3+\beta c_3=0$. Hence $\psi^{-1}(P)$ consists of a unique point.

For a plane $P\in Gr(3,5)$ represented by
\[\left( \begin{matrix}
1&0&a_2&a_3&0\\
0&1&b_2&b_3&0\\
0&0&c_2&c_3&1
\end{matrix}\right),\]
let us consider lines in $P$ represented by
\[\left( \begin{matrix}
1&0&a_2+\alpha c_2&a_3+\alpha c_2 & \alpha \\
0&1&b_2+\beta c_2&b_3+\beta c_3& \beta
\end{matrix}\right).\]
The equation $p_{12}=p_{03}$ gives a unique line $Z$ in $Y$ defined by $a_2+b_3+\alpha c_2+\beta c_3=0$. Hence $\psi^{-1}(P)$ consists of a unique point unless $c_2=c_3=a_2+b_3=0$. When the equations hold, $\psi^{-1}(P)=P^\vee\cong \PP^2$ is the space of lines in the plane $P$.

By repeating the same with all other charts, we find that there is a smooth quadric 3-fold
$$\Sigma=Gr(2,U_4)\cap H'\subset H'\cong \PP^4$$
where for $U_4=\langle e_0,e_1,e_2,e_3\rangle$, $Gr(2,U_4)\subset Gr(3,5)$ is the linear embedding: $U_2\mapsto U_2+\langle e_4\rangle$ and  $H'=\zero(p_{12}-p_{03})$ is the hyperplane in $\PP(\wedge^2U_4)\cong \PP^5$. Note that 
$\psi^{-1}(P)$ is a point for $P\notin \Sigma$ and $P^\vee$ for $P\in \Sigma$.
It is straightforward to prove that $\psi$ is the blowup along $\Sigma$, by using explicit local chart calculation. For instance, consider the local chart $(a_2,b_2,c_2,a_3,b_3,c_3,\lambda,\mu)$ of $Gr(1,3,5)$ represented by
\[\left( \begin{matrix}
1&\lambda&a_2+\lambda b_2+\mu c_2&a_3+\lambda b_3+\mu c_3&\mu\\
0&1&b_2&b_3&0\\
0&0&c_2&c_3&1
\end{matrix}\right)\]
whose first row represents the one dimensional subspace $V_1$ while all the three rows span a three dimensional subspace $V_3$. Let us consider lines in the plane $\PP V_3$ passing through $\PP V_1$ represented by
\[\left( \begin{matrix}
1&\lambda&a_2+\lambda b_2+\mu c_2&a_3+\lambda b_3+\mu c_3&\mu \\
0&\alpha&\alpha b_2+\beta c_2&\alpha b_3+\beta c_3& \beta
\end{matrix}\right).\]
The equation $p_{12}=p_{03}$ gives us $a_2+b_3=-\mu c_2$ and $c_3=\lambda c_2$, which define $F_1(Y)$. Certainly this is the blowup map $(c_2,\lambda,\mu)\mapsto (c_2,c_3,a_2+b_3)$ along $\Sigma=\zero(c_2,c_3,a_2+b_3)$ in local coordinates.
\end{proof}

\begin{lemm}\label{smoothness}
Let $F_1(Y)$ be the moduli space of lines in $Y$. Then $F_1(Y)$ is smooth.
\end{lemm}
\begin{proof}
As seen in Proposition \ref{sy551}, the reduced scheme $F_1(Y)_{\mathrm{red}}$ is irreducible. Let $L$ be a line in $Y$.
Note that $L$ is a locally complete intersection for all $L\subset Y$. From the nested normal bundle sequence $\ses{N_{L/Y}}{N_{L/G}}{N_{Y/G}|_L=\cO_L(1)}$, the expected dimension of $F_1(Y)$ is equal to $h^0(N_{L/Y})-h^1(N_{L/Y})=6$. But since each point $[L]$ in $F_1(Y)$ has dimension $6$ by Proposition \ref{sy551} again, $F_1(Y)$ is a locally complete intersection by \cite[Theorem 2.15]{KOL}. This implies that $F_1(Y)$ is Cohen-Macaulay. Also it is known that a generically reduced Cohen-Macaulay space is reduced (\cite[page 49-51]{LO2}).
Hence it is enough to prove that $h^1(N_{L/Y})=0$ for a single line $L$ in $Y$, which can be checked for the line $L$ represented by $\left( \begin{matrix}
1&0&0&0&0 \\
0&s&t&0&0
\end{matrix}\right)$. 
After all, we proved that $F_1(Y)=F_1(Y)_{\mathrm{red}}\cong\mathrm{bl}_{\Sigma} Gr(3,5)$ is smooth.
\end{proof}

\begin{prop}(\cite[\S 4.4]{iliev3})\label{sy552}
The Fano variety of planes $F_2(Y)$ consists of two components $F_2^{3,1}(Y)\sqcup F_2^{2,2}(Y)$. The first component $F_2^{3,1}(Y)$ is the blow-up of $\PP^4$ at a point $y_0$ and the second component $F_2^{2,2}(Y)$ is a smooth quadric 3-fold $\Sigma$.
\end{prop}
\begin{proof}
From the proof of the previous lemma, we find that the locus $F_2^{2,2}(Y)$ of $\sigma_{2,2}$-planes is the quadric 3-fold $\Sigma$.

Assigning the vertex gives a morphism $\psi:F_2^{3,1}(Y)\to \PP^4$. Let $y=[1:a_1:a_2:a_3:a_4]$ be the vertex of a $\sigma_{3,1}$-line represented by
\[\left(\begin{matrix} 1&a_1&a_2&a_3&a_4\\ 0&b_1&b_2&b_3&b_4 \end{matrix}\right).\]
The equation $p_{12}-p_{03}$ is a nonzero linear equation $b_3=a_1b_2-a_2b_1$ and thus a unique plane.

By calculating with all charts, we find that $\psi^{-1}(y)$ is a point if and only if $y=[0:0:0:0:1]=:y_0$ and $\psi^{-1}(y_0)$ is the dual $Gr(3,4)\cong \PP^4$ of planes in $\zero(y_4)\cong \PP^3$.

It is straightforward to check that $\psi$ is the blowup of $\PP^4$ at $y_0$ by explicit local calculation. 
For instance, let us consider the local chart $$\{([a_0:a_1:a_2:a_3:1],[c_0:c_1:c_2:c_3:c_4])|a_0c_0+a_1c_1+a_2c_2+a_3c_3+c_4=0\}$$ of $Gr(1,4,5)\subset Gr(1,5)\times Gr(4,5)\cong \PP^4\times \PP^{4*}$.
The linear space $\PP^3$ of a $\sigma_{3,1}(y,\PP^3)$-plane in $Y$ with the vertex $y=[a_0:a_1:a_2:a_3:1]$ is given by $a_1x_2-a_2x_1-a_0x_3+a_3x_0=0$, which comes from the relation $p_{12}-p_{03}=0$. Therefore the equation of $F_2^{3,1}(Y)$ is given by the condition
\[
\rank \left(\begin{matrix} a_3&-a_2&a_1&-a_0&0\\ c_0&c_1&c_2&c_3&c_4\end{matrix}\right)=\rank \left(\begin{matrix} a_3&-a_2&a_1&-a_0\\ c_0&c_1&c_2&c_3\end{matrix}\right)=1
\]
which implies that $F_2^{3,1}(Y)\cong \mathrm{bl}_0\CC^4$. By the same argument in Lemma \ref{smoothness}, one can show that the Hilbert scheme $F_2(Y)$ of planes is smooth and thus finish the proof.
\end{proof}

\subsection{Conics in $Y^5_5$}\label{birgeo}
Cooking up the geometric properties of the lines and planes in $Y_5^5$, we study the birational relation between $H_2(Y_5^5)$ and a projective model via the diagram \eqref{diahil} (Proposition \ref{birmodel}). The key point of the argument depends on \cite[Definition-Proposition 3.4]{CHK}.

Let $\cU$ be the universal subbundle on $Gr(4,5)$. Let 
\begin{equation}\label{ker}
\cK:= \mathrm{ker}\{\wedge^2\cU\subset \wedge^2 \cO^{\oplus 5}\to \cO \}\end{equation}
be the kernel of the composition map where the second arrow is $p_{12}-p_{03}$ (cf. \cite[Proposition B.6.1]{KPS16}). Note that $\cK$ is locally free because of the choice of the hyperplane $p_{12}-p_{03}=0$. Let $S(Y):=Gr(3, \cK)$ and then $S(Y)\subset S(G):= Gr(3,\wedge^2\cU)$ by its definition. Recall that the blow-up center $T(G)$ in Proposition \ref{chen} can be described in terms of the disjoint union of the flag varieties: $Gr(1,4,5)\sqcup Gr(3,4,5)$.

Let $V=\CC^5$ be a $5$-dimensional vector space. Note that the space $S(G)$ is given by an incident variety $$S(G)=\{(U,V_4)|U\subset \wedge^2 V_4\}\subset Gr(3,\wedge^2V)\times Gr(4, V).$$
The embedding map $T(G):=T^{3,1}(G)\sqcup T^{2,2}(G)\hookrightarrow S(G)$ is defined as follow.
\begin{itemize}
\item For $[(V_1,V_4)]\in T^{3,1}(G)$, $$(V_1,V_4)\mapsto (W, V_4)$$ where $W=\ker(\wedge^2V_4\twoheadrightarrow \wedge^2(V_4/V_1))(=V_1\wedge V_4)$ is the $3$-dimensional vector space.
\item For $[(V_3,V_4)]\in T^{2,2}(G)$, $$(V_3,V_4)\mapsto (\wedge^2V_3, V_4).$$
\end{itemize}

Let $T(Y):= S(Y)\cap T(G)$.
\begin{prop}
The intersection part $T(Y)$ is a disjoint union of two connected components: $T^{3,1}(Y)\sqcup T^{2,2}(Y)$ where
\begin{enumerate}
\item $T^{3,1}(Y)\cong F^{3,1}(Y)$ and
\item $T^{2,2}(Y)$ is isomorphic to a $\PP^1$-bundle over the quadric $3$-fold $\Sigma(= F^{2,2}(Y))$. 
\end{enumerate}
\end{prop}
\begin{proof}
The first part is clear. The second part comes from the direct computation from a composition digram
\[
T^{2,2} (Y)\subset Gr(3,4,5)\stackrel{p}{\longrightarrow} Gr(3,5)
\]
where the image $p(T^{2,2}(Y))=\Sigma$ is the smooth quadric $Gr(2,V_4^0)\cap H'$ (cf., Proposition \ref{sy552}). Here $V_4^0=\mathrm{span}\langle e_0,e_1,e_2,e_3\rangle$ and $H'=\mathrm{zero}(p_{12}-p_{03})$ is the hyperplane in $\PP(\wedge^2 V_4^0)$. This can be proved directly by computing each affine chart. For instance, let $(V_3,V_4)\in Gr(3,4,5)$. Let $P(V_3)=[\PP(v_1,v_2,v_3)]\in Gr(3,5)$ be the plane represented by
\[\left( \begin{matrix}
1&0&a_2&a_3&0\\
0&1&b_2&b_3&0\\
0&0&c_2&c_3&1
\end{matrix}\right).\]
Then, $\wedge^2(\langle v_1,v_2,v_3\rangle)\subset K_{[V_4]}$ if and only if $c_2=c_3=a_2+b_3=0$. Doing the same computation in other chart, we obtain the result.
\end{proof}
In fact, one can show that the projection map of the intersection part $T(Y)\subset T(G)=Gr(1,4,5)\sqcup Gr(3,4,5) \lr Gr(4,5)$ has a bundle structure over its base.
\begin{prop}\label{diff}
 The intersection part $T(Y)$ is a $\PP^1\sqcup \PP^1$-bundle over the Grassmanian $Gr(3,4)$ linearly embedded in $Gr(4,5)$.
\end{prop}
\begin{proof}
Recall that $V$ be a $5$-dimensional vector space with a fixed basis $\{e_0,e_1,e_2,e_3,e_4\}$. For $U_4=\langle e_0,e_1,e_2,e_3\rangle$, the linear embedding $Gr(3,U_4)\subset Gr(4,V)$ is given by $W \mapsto W+\langle e_4\rangle$ for $[W]\in Gr(3,4)$.
Let $[V_4] \in Gr(4,5)$ and $\Omega$ be a skew-symmetric 2-form on $V$ induced from $p_{12}-p_{03}$. More exiplcitly, $\Omega$ is a 2-form on $V$ represented by the skew-symmetric matrix $(x_{ij})\in \PP(\wedge^2V^*)$ whose entries are all zero except for $x_{12}= -x_{21}=x_{30}=-x_{03}=1$. Clearly, $\rank\Omega =4$ and thus $\rank\Omega|_{V_4}\neq 0$.
If the $\rank\Omega|_{V_4}$=4, then there is no vector $v\in \ker\Omega|_{V_4}$. Thus there is no $\sigma_{3,1}$-plane in the fiber of $T(Y_5^5)$ over $V$. Also there is no $3$-dimemsnional subspace $V_3 \subset V_4$ such that $\Omega|_{V_3}$=0 since it forces the $\rank\Omega|_{V_4}$ to be equal or less than 2. Therefore the fiber is empty when the $\rank\Omega|_{V_4}=4$. Consider the case that the $\rank\Omega|_{V_4}=2$. If $V_4\cap (\ker\Omega=\langle e_4 \rangle)=\langle 0 \rangle$, then the natural morphism $V_4 \mapsto V/\langle e_4 \rangle$ is an isomorphism and the $\rank\Omega|_{V_4}$=4 since the 2-form $\Omega$ descents  to a rank $4$ form on $V/\langle e_4 \rangle$, which is a contradiction. Thus $\langle e_4 \rangle \subset V_4$. Conversely, if $\langle e_4 \rangle \subset V_4$, then the $\rank\Omega \leq 2$ since $\langle e_4 \rangle \subset \ker \Omega|_{V_4}$, therefore it is equal to $2$. Hence the $\rank\Omega|_{V_4}=2$ if and only if $V_4 \in Gr(3,V_4)$ which is linearly embedded in $Gr(4,V)$. And the set of $\sigma_{3,1}$-planes over $V$ clearly equals to $\PP(\ker \Omega|_{V_4})\cong \PP^1 \subset Gr(1,V_4)=\PP^3$.

Consider a $\sigma_{2,2}$-plane over $[V_4]\in Gr(4,5)$ represented by 3-dimensional subspace $V_3\subset V_4$, i.e. $\Omega|_{V_3}=0$. $\Omega|_{V_4}$ descents to a rank $2$ form $\overline{\Omega}|_{V_4}$ on $V_4/\ker\Omega|_{V_4}$. Hence, if $\dim V_3/(\ker\Omega|_{V_4} \cap V_3)=2$, the natural morphism $V_3/(\ker\Omega \cap V_3) \mapsto V_4/\ker\Omega|_{V_4}$ is an isomorphism and therefore the $\rank\Omega|_{V_3}$=$\rank\overline{\Omega}|_{V_4}=2$, which is a contradiction. Therefore $\dim V_3/(\ker\Omega|_{V_4} \cap V_3)=1$, which is equivalent to that $\ker\Omega|_{V_4} \subset V_3$. Therefore the set of $\sigma_{2,2}$-planes corresponds to $\PP((V_4/\ker \Omega|_{V_4})^*) \cong \PP^1 \subset \PP(V_4^*)=Gr(3,4)$. 
\end{proof}

For the later use, we need to confirm that the intersection part $T(Y)= S(Y)\cap T(G)$ is a scheme theoretic one (cf. \cite[Lemma 5.1]{LLi}). Let us denote $T_{X,x}$ by the tangent space of a smooth variety $X$ at a closed point $x$.
\begin{lemm}\label{clean}
\[T_{T(Y), V}= T_{S(Y),V}\cap T_{T(G),V}\]
for all $V\in T(Y)$.
\end{lemm}
\begin{proof}
Consider the tangent bundle sequences
\begin{equation}\label{bundle}
\xymatrix{
0\ar[r]&T_{T(Y), V}\ar[r]\ar@{^{(}->}[d]&T_{T(G),V}\ar[r]\ar@{^{(}->}[d]&N_{T(Y)/T(G),V}\ar[r]\ar[d]^{q}&0\\
0\ar[r]&T_{S(Y), V}\ar[r]&T_{S(G),V}\ar[r]&N_{S(Y)/S(G),V}\ar[r]&0.
}
\end{equation}
To prove the claim, it is suffice to prove that the induced map $q$ in \eqref{bundle} is an isomorphism. But the direct computation says that there exists a commutative diagram
\begin{equation}\label{iden}
\xymatrix{
N_{T(Y)/T(G),V}\ar^{\cong}[r]\ar[d]^q&H^0(\cO_H(1))\ar@{=}[d]\\
N_{S(Y)/S(G),V}\ar^{\cong}[r]& \Hom (V_3, \CC).
}
\end{equation}
for $V=[(V_1, V_4)]\in T^{3,1}(Y)=F^{3,1}(Y)$ (similarly for $T^{2,2}(Y)$). Here the plane $H=\PP(V_3)$ is determined by $V$ (That is, $V_3:= \mathrm{ker}(\wedge^2V_4\lr \wedge^2(V_4/V_1))$). The first horizontal isomorphism in \eqref{iden} comes from the nested normal bundle sequence $$\ses{N_{H/Y}}{N_{H/G}}{N_{Y/G}|_H\cong \cO_H(1)},$$ where $h^1(N_{H/Y})=0$ by Proposition \ref{sy552}. The second horizontal isomorphism comes from the identifications
\begin{align*}
N_{S(Y)/S(G),V}=N_{Gr(3,5)/Gr(3,6),V}&\cong \Hom(V_3, \wedge^2V_4/V_3)/\Hom(V_3, \cK_{[V_4]}/V_3)\\
&\cong \Hom (V_3, \CC)
\end{align*}
by the equation \eqref{ker} and thus we finished the proof of the claim.
\end{proof}
\begin{rema}
Using the local chart obtained by Proposition \ref{diff},  one can directly show that the scheme theoretic intersection is the same as the set-theoretic one. For detail, see \cite{Lee18}.
\end{rema}
\begin{prop}\label{birmodel}
Let $H_2(Y)$ be the Hilbert scheme of conics in $Y=Y_5^5$. 
Then $H_2(Y)$ is obtained from $S(Y):=Gr(3, \cK)$ by a blow-down followed by a blow-up:
\begin{equation}
\xymatrix{
&\widetilde{S}(Y)\ar[rd]\ar[ld]&\\
S(Y)&& H_2(Y).
}
\end{equation}
In particular, $H_2(Y)$ is smooth, irreducible variety of dimension $10$.
\end{prop}
\begin{proof}
By Lemma \ref{clean}, the blown-up space $\widetilde{S}(Y)$ is the proper transformation of $S(Y)$ through the blow-up morphism $\Xi:H_2(Gr(2,\cU))\longrightarrow S(Gr(2,5))$ in Proposition \ref{chen} (\cite[Lemma 5.1]{LLi}). One can easily check that the restricted normal bundle of the original one is $\cO(-1)$ (cf. \cite[Proposition 3.6]{CHK}). Applying Fujiki-Nakano criterion (\cite{FN71}), the blow-down space is smooth. Therefore the Hilbert scheme $H_2(Y)$ is smooth whenver $H_2(Y)$ is reduced and irreducible. Irreducibility of $H_2(Y)$ directly comes from the diagram \eqref{diahil}. Also, $H_2(Y)$ is reduced since we can repeat the argument in Lemma \ref{smoothness}. Thus we finish the proof.
\end{proof}

\section{Fano 4-fold $Y_5^4$}\label{sy45}
In this section we let $Y=Y_5^4$ denote the Fano 4-fold defined as the intersection of $G=Gr(2,5)$ with two general hyperplanes $H_1$, $H_2$ in $\PP(\wedge^2\CC^5)= \PP^9$. Let $p_{ij}$ denote the Pl\"ucker coordinates. For explicit calculations, we will let
$$H_1=\{p_{12}-p_{03}=0\},\quad H_2=\{p_{13}-p_{24}=0\}.$$
The results on lines and planes in $Y$ are due to Todd \cite{Todd}. We include elementary proofs of the results for reader's convenience. The result on conics seems new.

\subsection{Fano varieties of lines and planes in $Y_5^4$}
We first recall the results on the Fano variety of planes and lines in $Y$.
The results in this subsection are due to Todd \cite{Todd}.

\begin{lemm} \label{p1} \cite{Todd}
There exists a unique $\sigma_{2,2}$-plane in $Y$, i.e.
there exists a unique plane $\Pi\subset \PP^4$ such that the dual variety $\Pi^\vee$ of
lines in $\Pi$ is contained in $Y$.
\end{lemm}

\begin{proof}
Consider the open chart
$$ \left(
\begin{matrix}
1 & 0& {a_2} & {a_3} & {a_4} \\
0 & 1& {b_2} & {b_3} & {b_4}
\end{matrix}
\right)$$
of $Gr(2,5)$. Then $p_{12}-p_{03}= -a_2-b_3$ and $p_{13}-p_{24}=-a_3-a_2b_4+a_4b_2$.
Finding $\Pi$ is the same as finding a pair of independent linear equations in $x_0,\cdots,x_4$ such that whenever
both $(1,0,a_2,a_3,a_4)$ and $(0,1,b_2,b_3,b_4)$ satisfy the two equations,
we should have $a_2=-b_3$ and $a_3=a_4b_2-a_2b_4$.
It is straightforward to see that the only such pair of linear equations is $x_2=0=x_3$. By the same argument with other charts,
we find that $\zero(x_2,x_3)$ is the unique plane $\Pi$ with $\Pi^\vee\subset Y.$
\end{proof}
\begin{rema}
The plane $\Pi$ in Lemma \ref{p1} has a crucial role for the structure of the Fano variety $Y^4_5$ (\cite[\S 3]{Pro}, \cite[\S 3]{DIM} and \cite{Fuj17}).
\end{rema}
Recall that a $\sigma_{3,1}$-plane is the set of lines in $\PP^3\subset \PP^4$ passing through a point $p$, which is called the \emph{vertex} of the cycle.
\begin{lemm}\label{p2} \cite{Todd}
There exists a one-dimensional family of $\sigma_{3,1}$-planes in $Y$ whose vertices form a smooth conic $C_0$ in $\Pi$. There is no other plane in $Y$.
\end{lemm}
\begin{proof}
Let $(1, a_1,a_2,a_3,a_4)$ be the vertex of a $\sigma_{3,1}$-plane. Then a line in the plane is represented by
$$\left(\begin{matrix} 1&a_1&a_2&a_3&a_4\\ 0&b_1&b_2&b_3&b_4\end{matrix}\right).$$
Since $p_{12}-p_{03}=a_1b_2-a_2b_1-b_3$ and $p_{13}-p_{24}=a_1b_3-a_3b_1-a_2b_4+a_4b_2$, these two linear equations in $(b_1,b_2,b_3,b_4)$ are parallel if and only if $$\rank \left(\begin{matrix} a_2&-a_1&1&0\\ -a_3&a_4&a_1&-a_2\end{matrix}\right)=1.$$
This holds if and only if $a_2=a_3=0$ and $a_4+a_1^2=0$. The first equation says the vertex lies in the plane $\Pi$ in Lemma \ref{p1} and the second says the locus of the vertices is a smooth conic in $\Pi$. By the same argument with other charts, we obtain the lemma.
\end{proof}

\begin{coro}
The Fano variety of planes in $Y$ is $C_0\sqcup \{\Pi\}$.
\end{coro}

\begin{prop}\label{l1}\cite{Todd}
Let $H_1(Y)$ denote the Hilbert scheme of lines in $Y=Y^4_5$. Then $H_1(Y)$ is isomorphic to the blowup of $\PP^4$ along the smooth conic $C_0\subset \Pi$ in Lemma \ref{p2}.
\end{prop}
\begin{proof}
Recall that a line $Z$ in $G=Gr(2,5)$ is the set of lines in a plane $P$ in $\PP^4$ passing through a point $p\in P$, which we call the vertex of the line $Z$. Assigning the vertex to a line in $V$ gives a morphism
$$\psi:H_1(Y)\subset H_1(G)=Gr(1,3,5)\lra  Gr(1,5)=\PP^4.$$

By the proof of Lemma \ref{p2}, $y=(1,a_1,a_2,a_3,a_4)\in \PP^4$, if $v\notin C_0$,
the $\sigma_{3,1}$-cycle $\sigma_{3,1}(v)$ with vertex $v$ intersects with $Y$ along a line. If $v\in C_0$,  $\sigma_{3,1}(v)$ intersects with $Y$ along a plane $P(v)$ in $Y$. Hence $\psi^{-1}(y)$ is a point for $y\notin C_0$ and $\PP^2$ for $y\in C_0$.

By explicit local chart calculation as in the proof of Proposition \ref{sy551}, it is again straightforward to show that $\psi$ is the blowup map along $C_0$. Also, by the same argument in Lemma \ref{smoothness}, one can show that the Hilbert scheme $H_1(Y)$ is smooth and thus finish the proof.
\end{proof}

\begin{prop}\label{l2} Let $C^\vee_0\subset H_1(Y)$ denote the dual conic of tangent lines to $C_0$ in $\Pi$.
Let $Z\in H_1(Y)$ be a line in $Y$. Then the normal bundle $N_{Z/Y}$ is $\cO_{Z}^{\oplus 2}\oplus \cO_Z(1)$ if $Z\notin C_0^\vee$ and $\cO_{Z}(-1)\oplus \cO_Z(1)^{\oplus 2}$ if $Z\in C_0^\vee$.
\end{prop}
\begin{proof}
Let $Z=(1,a_1,a_2,a_3,a_4)$ be a line in $\Pi^\vee\subset Y$. Suppose the line $Z$
is represented by
\[ \left(
\begin{matrix}
1 & a_1 & {a_2} & {a_3} & {a_4} \\
0 & b_1 & {b_2} & {b_3} & {b_4}
\end{matrix}
\right)
\]
with $(a_1,b_1)$ the homogeneous coordinates for $Z$. Then $a_2,a_3,b_2,b_3$ are the coordinates for the normal bundle $N_{\Pi^\vee/G}|_Z$ with $a_2, a_3$ giving $\cO_Z^{\oplus 2}$ while $b_2,b_3$ giving $\cO_Z(1)^{\oplus 2}$. The two equations $p_{12}-p_{03}$ and $p_{13}-p_{24}$ give us homomorphisms
$$\cO_Z^{\oplus 2}\oplus \cO_Z(1)^{\oplus 2}\lra \cO(1)^{\oplus 2}, \quad (a_2,a_3,b_2,b_3)\mapsto (a_1b_2-a_2b_1-b_3, a_1b_3-a_3b_1-a_2b_4+a_4b_2).$$
If $a_4+a_1^2\ne 0$, then the kernel of this homomorphism is $\cO_Z^{\oplus 2}.$ If $a_4+a_1^2=0$, then the kernel is $\cO_Z(-1)\oplus \cO_Z(1)$. The equation $a_4+a_1^2=0$ defines the conic $C_0$.
Since the normal bundle $N_{Z/\Pi^\vee}$ is $\cO_Z(1)$, the exact sequence $0\to N_{Z/\Pi^\vee}\to N_{Z/Y}\to N_{\Pi^\vee/Y}|_Z\to 0$ splits to $N_{Z/Y}\cong \cO_Z(1)\oplus N_{\Pi^\vee/Y}|_Z$. So we have the proposition.
\end{proof}

\subsection{Conics in $Y^4_5$}\label{birgeo}
In this subsection, we study the birational relation between $H_2(Y_5^4)$ and a projective model by using the diagram \eqref{diahil}.
\begin{prop}\label{birmodel2}
Let $H_2(Y)$ be the Hilbert scheme of conics in $Y=Y_5^4$. Let $\cU$ be the universal subbundle on $Gr(4,5)$. Let \[\cK:= \mathrm{ker}\{\wedge^2\cU\subset \wedge^2 \cO^{\oplus 5}\to \cO^{\oplus 2} \}\] be the kernel of the composition map where the second arrow is $(p_{12}-p_{03},p_{13}-p_{24})$.
Then $H_2(Y)$ is obtained from $S(Y):=Gr(3, \cK)$ by a blow-down followed by a blow-up:
\begin{equation}
\xymatrix{
&\widetilde{S}(Y)\ar[rd]\ar[ld]&\\
S(Y)&& H_2(Y).
}
\end{equation}
In particular, $H_2(Y)$ is an irreducible, smooth variety of dimension $7$.
\end{prop}
\begin{proof}
The proof can be done by the parallel way of that of Proposition \ref{birmodel}.
\end{proof}

\begin{rema}\label{diff4}
As the $5$-fold case in Proposition \ref{diff}, the blow-up center $T(Y)\subset T(G)=Gr(1,4,5)\sqcup Gr(3,4,5) \lr Gr(4,5)$ in $S(Y)$ has a bundle structure over its base. In fact,
the intersection $T(Y)$ is a disjoint $\PP^0\sqcup \PP^0$-bundle over $\PP^1$ which is linearly embedded in $Gr(4,5)$. Each fiber of $T(Y)$ over $[V_4]\in \PP^1$ is the following: \[ \{((\ker(\Omega|_{V_4})\cap\ker(\Omega|_{V_4}'),V_4), ((\ker (\Omega|_{V_4})+\mathrm{ker}(\Omega|_{V_4}')),V_4)\} \subset Gr(1,4,5)\sqcup Gr(3,4,5). \] Here $\Omega$ (resp. $\Omega')$ is the skew-symmetric $2$-form on V induced from $p_{12}-p_{03}$ (resp. $p_{13}-p_{24}$).
\end{rema}

\section{Fano 3-fold $Y_5^3$}\label{sy35}
The method of proofs in previous sections can be applied to the Fano 3-fold, which enables us to re-prove the well-known results about the space of lines and conics.
For explicit calculations, we let
$$H_1=\{p_{12}-p_{03}=0\},\quad H_2=\{p_{13}-p_{24}=0\}, \quad H_3=\{p_{14}-p_{02}=0\}.$$

\begin{prop}
The Hilbert scheme $F_1(Y)$ of lines is isomorphic to a projective plane $\PP^2$.
\end{prop}
\begin{proof}
Consider the vertex map $F_1(Y)\subset Gr(1,3,5)\lr Gr(1,5)$. Then for a vertex $p=[a_0:a_1:a_2:a_3:a_4]$, a line in $Gr(2,5)$ with the vertex $p$ is contained in $Y$ if and only if
\[
\mathrm{rank} \left(
\begin{matrix}
a_3 & 0& a_2\\
a_2 & a_3 & a_4\\
-a_1&-a_4&-a_0\\
-a_0&-a_1&0\\
0&a_2&-a_1
\end{matrix}
\right)\leq 2.
\]
Then the condition says that the defining ideal of the image $\mathrm{Im}(F_1(Y))$ in $Gr(1,5)=\PP^4$ is given by
\begin{multline*}
\langle a_1 a_2 a_3+a_0 a_3^2-a_2^2 a_4+a_3 a_4^2,a_2^3-a_1a_3^2-a_2 a_3 a_4,a_1 a_2^2+a_0 a_2 a_3-a_1 a_3 a_4,a_0 a_2^2+a_1^2 a_3,\\
a_1^2 a_2+a_0 a_1 a_3+a_0 a_2 a_4,a_0 a_1 a_2+a_0^2 a_3+a_1^2 a_4+a_0 a_4^2,a_1^3-a_0^2 a_2+a_0 a_1 a_4\rangle.
\end{multline*}
It is well-known that this is isomorphic to a $\PP^2$, which is a projected Veronese surface (\cite[Theorem 1.1]{Par07}). Note that $\mathrm{Im}(F_1(Y))\cong \PP^2$ contains the distinguished conic $C_0$ in Proposition \ref{l1}. Thus $F_1(Y)=bl_{C_0}\mathrm{Im}(F_1(Y))\cong \PP^2$.
\end{proof}

\begin{prop}
The Hilbert scheme $H_2(Y)$ of conics is isomorphic to $Gr(4,5)\cong \PP^4$.
\end{prop}
\begin{proof}
By the argument in the proof of Lemma \ref{p1}, one easily see that $Y$ does not contain any planes. Furthermore, 
let \[\cK:= \mathrm{ker}\{\wedge^2\cU\subset \wedge^2 \cO^{\oplus 5}\to \cO^{\oplus 3} \}\] be the kernel where $\cU$ is the universal bundle over $Gr(4,5)$ and the second arrow is $(p_{12}-p_{03},p_{13}-p_{24},p_{14}-p_{02})$.
Then one can easily see that $\cK$ is a rank $3$ vector bundle. And thus $H_2(Y)_{\mathrm{red}}\cong S(Y)\cong Gr(4,5)$. By the same argument in Lemma \ref{smoothness}, one can show that $H_2(Y)_{\mathrm{red}}=H_2(Y)$.
\end{proof}

\bibliographystyle{amsplain}

\end{document}